\documentclass[11pt]{amsart}
\usepackage{amscd}
\usepackage{amsmath}
\usepackage{amsxtra}
\usepackage{amsfonts}
\usepackage{amssymb}
\usepackage{graphicx}

\newtheorem{theorem}{Theorem}[section]
\newtheorem{corollary}[theorem]{Corollary}

\newtheorem{proposition}[theorem]{Proposition}

\newtheorem{conjecture}[theorem]{Conjecture}
\theoremstyle{definition}
\newtheorem{definition}[theorem]{Definition}
\newtheorem{remark}[theorem]{Remark}

\newtheorem{example}[theorem]{Example}
\theoremstyle{remark}

\renewcommand{\theclaim}{\textup{\theclaim}}

\newtheorem*{acknowledgements}{Acknowledgements}

\numberwithin{equation}{section}

\def\openone

{\mathchoice

{\hbox{\upshape \small1\kern-3.3pt\normalsize1}}

{\hbox{\upshape \small1\kern-3.3pt\normalsize1}}

{\hbox{\upshape \tiny1\kern-2.3pt\SMALL1}}

{\hbox{\upshape \Tiny1\kern-2pt\tiny1}}}

\makeatletter

\newbox\ipbox

\newcommand{\ip}[2]{\left\langle #1\, , \,#2\right\rangle}
\newcommand{\diracb}[1]{\left\langle #1\mathrel{\mathchoice

{\setbox\ipbox=\hbox{$\displaystyle \left\langle\mathstrut
#1\right.$}

\vrule height\ht\ipbox width0.25pt depth\dp\ipbox}

{\setbox\ipbox=\hbox{$\textstyle \left\langle\mathstrut
#1\right.$}

\vrule height\ht\ipbox width0.25pt depth\dp\ipbox}

{\setbox\ipbox=\hbox{$\scriptstyle \left\langle\mathstrut
#1\right.$}

\vrule height\ht\ipbox width0.25pt depth\dp\ipbox}

{\setbox\ipbox=\hbox{$\scriptscriptstyle \left\langle\mathstrut
#1\right.$}

\vrule height\ht\ipbox width0.25pt depth\dp\ipbox}

}\right. }

\newcommand{\dirack}[1]{\left. \mathrel{\mathchoice

{\setbox\ipbox=\hbox{$\displaystyle \left.\mathstrut
#1\right\rangle$}

\vrule height\ht\ipbox width0.25pt depth\dp\ipbox}

{\setbox\ipbox=\hbox{$\textstyle \left.\mathstrut
#1\right\rangle$}

\vrule height\ht\ipbox width0.25pt depth\dp\ipbox}

{\setbox\ipbox=\hbox{$\scriptstyle \left.\mathstrut
#1\right\rangle$}

\vrule height\ht\ipbox width0.25pt depth\dp\ipbox}

{\setbox\ipbox=\hbox{$\scriptscriptstyle \left.\mathstrut
#1\right\rangle$}

\vrule height\ht\ipbox width0.25pt depth\dp\ipbox}

} #1\right\rangle}

\newcommand{\beq}{\begin{equation}}

\newcommand{\eeq}{\end{equation}}

\newcommand{\Ds}{\mathsf{D}}

\newcommand{\cj}[1]{\overline{#1}}

\newcommand{\bz}{\mathbb{Z}}

\newcommand{\br}{\mathbb{R}}
\newcommand{\bc}{\mathbb{C}}

\def\blfootnote{\xdef\@thefnmark{}\@footnotetext}

\renewcommand{\mod}{\operatorname{mod}}

\input xy
\xyoption{all}
\usepackage{amssymb}

\def\T{\mathcal{T}}

\def\F{\mathcal{F}}

\def\-{^{-1}}

\begin{document}

\title[Local translations associated to spectral sets]{Local translations associated to spectral sets}
\author{Dorin Ervin Dutkay}

\address{[Dorin Ervin Dutkay] University of Central Florida\\
	Department of Mathematics\\
	4000 Central Florida Blvd.\\
	P.O. Box 161364\\
	Orlando, FL 32816-1364\\
U.S.A.\\} \email{Dorin.Dutkay@ucf.edu}

\author{John Haussermann}

\address{[John Haussermann] University of Central Florida\\
	Department of Mathematics\\
	4000 Central Florida Blvd.\\
	P.O. Box 161364\\
	Orlando, FL 32816-1364\\
U.S.A.\\} \email{jhaussermann@knights.ucf.edu}

\thanks{} 

\begin{abstract}
In connection to the Fuglede conjecture, we study groups of local translations associated to spectral sets, i.e., measurable sets in $\br$ or $\bz$ that have an orthogonal basis of exponential functions. We investigate the connections between the groups of local translations on $\bz$ and on $\br$ and present some examples for low cardinality. We present some relations between the group of local translations and tilings. 
\end{abstract}
\maketitle \tableofcontents

\section{Introduction}
In his study of commuting self-adjoint extensions of partial differential operators, Fuglede proposed the following conjecture \cite{Fug74}: 
\begin{conjecture}
Denote by $e_\lambda(x)=e^{2\pi i\lambda\cdot x}$, $\lambda,x\in\br^d$. Let $\Omega$ be a Lebesgue measurable subset of $\br^d$ of finite positive measure. There exists a set $\Lambda$ such that $\{e_\lambda : \lambda\in\Lambda\}$ is an orthogonal basis in $L^2(\Omega)$ if and only if $\Omega$ tiles $\br^d$ by translations. 
\end{conjecture}

\begin{definition}\label{def0.1}
Let $\Omega$ be a Lebesgue measurable subset of $\br^d$ of finite, positive measure. We say that $\Omega$ is a {\it spectral} set if there exists a set $\Lambda$ in $\br^d$ such that $\{e_\lambda : \lambda\in\Lambda\}$ is an orthogonal basis in $L^2(\Omega)$. In this case, $\Lambda$ is called a {\it spectrum} for $\Omega$. 

We say that $\Omega$ {\it tiles $\br^d$ by translations} if there exists a set $\T$ in $\br^d$ such that the sets $\Omega+t$, $t\in\T$ form a partition of $\br^d$, up to measure zero. 

\end{definition}

Terrence Tao \cite{Tao04} has proved that spectral-tile implication in the Fuglede conjecture is false in dimensions $d\geq 5$ and later both directions were disproved in dimensions $d\geq3$, see \cite{KM06}. At the moment the conjecture is still open in both directions for dimensions 1 and 2. In this paper, we will only focus on dimension $d=1$.

Recent investigations have shown that the Fuglede can be reduced to analogous statements in $\bz$, see \cite{DL13}. 

\begin{definition}\label{def0.2}
Let $A$ be a finite subset of $\bz$, $|A|=N$. We say that $A$ is {\it spectral } if there exists a set $\Gamma$ in $\br$ such that $\{e_\gamma : \gamma\in\Gamma\}$ is an orthogonal basis in $l^2(A)$, equivalently, the matrix 
\begin{equation}
\frac1{\sqrt{N}}\left(e^{2\pi i \gamma a}\right)_{\gamma\in\Gamma,a\in A}
\label{eq0.2.1}
\end{equation}
is unitary. This matrix is called {\it the Hadamard matrix associated to the pair} $(A,\Gamma)$.

We say that $A$ {\it tiles $\bz$ by translations} if there exists a set $\T$ in $\bz$ such that the sets $A+t$, $t\in\T$ forms a partition of $\bz$. 
\end{definition}

The Fuglede conjecture for $\bz$ can be formulated as follows
\begin{conjecture}\label{conjz}
A finite subset $A$ of $\bz$ is spectral if and only if it tiles $\bz$ by translations.
\end{conjecture}

As shown in \cite{DL13} the tile-spectral implications for $\br$ and $\bz$ are equivalent:
\begin{theorem}
Every bounded tile in $\br$ is spectral if and only if every tile in $\bz$ is spectral.
\end{theorem}

It is not clear if the spectral-tile implications for $\br$ and $\bz$ are equivalent. It is known that, if every spectral set in $\br$ is a tile, then every spectral set in $\bz$ is a tile in $\bz$. It was shown in \cite{DL13}, that the reverse also holds under some extra assumptions. 

\begin{theorem}
Suppose every bounded spectral set $\Omega$ in $\br$, of Lebesgue measure $|\Omega|=1$, has a rational spectrum, $\Lambda\subset\mathbb Q$. Then the spectral-tile implications for $\br$ and for $\bz$ are equivalent, i.e., every bounded spectral set in $\br$ is a tile if and only if every spectral set in $\bz$ is a tile. 
\end{theorem}

All known examples of spectral sets of Lebesgue measure 1 have a rational spectrum. There is another, stronger variation of the spectral-tile implication in $\bz$ which is equivalent to the spectral-tile implication in $\br$. 
\begin{definition}
We say that a set $\Lambda$ has {\it period} $p$ if $\Lambda+p=\Lambda$. The smallest positive $p$ with this property is called the {\it minimal period}. 
\end{definition}

\begin{theorem}\cite{DL13}
The following statements are equivalent
\begin{enumerate}
	\item Every bounded spectral set in $\br$ is a tile. 
	\item For every finite union of intervals $\Omega = A+[0,1]$ with $A\subset {\mathbb Z}$,  if $\Lambda$ is a spectrum of $\Omega$ with minimal period $\frac{1}{N}$, then $\Omega$ tiles ${\mathbb R}$ with a tiling set ${\mathcal T}\subset N{\mathbb Z}$.
\end{enumerate}
\end{theorem}

We should note here that for a finite set $A$ in $\bz$, the set $A+[0,1]$ is spectral in $\br$ if and only if $A$ is spectral in $\bz$. Also, if $\Gamma$ is a spectrum for $A$ then $\Lambda:=\Gamma+\bz$ is a spectrum of $A+[0,1]$, therefore $\Lambda$ has period 1 and the minimal period will have to be of the form $\frac1N$. See \cite{DL13} for details.

The spectral property of a set, in either $\br$ or $\bz$, can be characterized by the existence of a certain unitary group of local translations. We will describe this in the next section and present some properties of these groups. In Theorem \ref{th2.6} we give a characterization of spectral sets in $\bz$ in terms of the existence of groups of local translations or of a local translation matrix. In Proposition \ref{pr1.5} we establish a formula that connects the local translation matrix for a spectral set $A$ in $\bz$ and the group of local translations for the spectral set $A+[0,1]$ in $\br$. Proposition \ref{pr3.9} shows how one can look for tiling sets for the spectral set $A$ using the local translation matrix. Proposition \ref{pr3.11} shows that the rationality of the spectrum is characterized by the periodicity of the group of local translations.

%
%
%
%

\section{Local translations}
In this section we introduce the unitary groups of local translations associated to spectral sets. These are one-parameter groups of unitary operators on $L^2(\Omega)$, for subsets $\Omega$ of $\br$, or on $l^2(A)$, for subsets of $\bz$, which act as translations on $\Omega$ or $A$ whenever such translations are possible. The existence of such groups was already noticed earlier by Fuglede \cite{Fug74} and \cite{Ped87}. They were further studied in \cite{DJ12u}. The idea is the following: the existence of an orthonormal basis $\{e_\lambda: \lambda\in\Lambda\}$ allows the construction of the Fourier transform from $L^2(\Omega)$ to $l^2(\Lambda)$. On $l^2(\Lambda)$ one has the unitary group of modulation operators, i.e., multiplication by $e_t$, or the diagonal matrix with entries $e^{2\pi i t\lambda}$, $\lambda\in\Lambda$. Conjugating via the Fourier transform we obtain the unitary group of local translations. For further information on local translations and connections to self-adjoint extensions and scattering theory, see \cite{JPT12a,JPTa,JPTc,JPTd}.

\begin{definition}\label{def3.3} 
Let $\Omega$ be a bounded Borel subset of $\br$. 
A {\it unitary group of local translations} on $\Omega$ is a strongly continuous one parameter unitary group $U(t)$ on $L^2(\Omega)$ with the property that for any $f\in L^2(\Omega)$ and any $t\in\br$,
\begin{equation}
(U(t)f)(x)=f(x+t)\mbox{ for a.e }x\in \Omega\cap(\Omega-t)
\label{eq3.3.1}
\end{equation}

If $\Omega$ is spectral with spectrum $\Lambda$, we define the Fourier transform $\mathcal F:L^2(\Omega)\rightarrow l^2(\Lambda)$ 
\begin{equation}\label{eq3.3.2}
\F f=\left(\ip{f}{\frac{1}{\sqrt{|\Omega|}}e_\lambda}\right)_{\lambda\in\Lambda},\quad(f\in L^2(\Omega)).
\end{equation}
We define the unitary group of local translations associated to $\Lambda$ by 
\begin{equation}
U_\Lambda(t)=\F^{-1}\hat U_\Lambda(t)\F\mbox{ where }\hat U_\Lambda(t)(a_\lambda)_{\lambda\in\Lambda}=(e^{2\pi i\lambda t}a_\lambda)_{\lambda\in\Lambda},\quad ((a_\lambda)\in l^2(\Lambda)).
\label{eq3.3.3}
\end{equation}
\end{definition}

\begin{proposition}\label{pr3.4}
With the notations in Definition \ref{def3.3}
\begin{equation}
U_\Lambda(t)e_\lambda=e_\lambda(t)e_\lambda,\quad(t\in\br,\lambda\in\Lambda)
\label{eq3.4.1}
\end{equation}
Thus the functions $e_\lambda$ are the eigenvectors for the operators $U(t)$ corresponding to the eigenvalues $e^{2\pi i \lambda t}$ with multiplicity one.
\end{proposition}

\begin{proof}
Clearly $\F e_\lambda=\sqrt{|\Omega|}\delta_\lambda$ for all $\lambda\in\Lambda$. The rest follows from a simple computation. 
\end{proof}
\begin{theorem}\label{th4.1}
Let $\Omega$ be a bounded Borel subset of $\br$. Assume that $\Omega$ is spectral with spectrum $\Lambda$. Let $U_\Lambda$ be the associated unitary group as in Definition \ref{def3.3}. Then $U:=U_\Lambda$ is a unitary group of local translations.

\end{theorem}

 In the particular case when $\Omega$ is a finite union of intervals the converse also holds:
\begin{theorem}\label{th3.1}\cite{DJ12u}
The set $\Omega=\cup_{i=1}^n(\alpha_i,\beta_i)$ is spectral if and only if there exists a strongly continuous one parameter unitary group $(U(t))_{t\in\br}$ on $L^2(\Omega)$ with the property that, for all $t\in\br$ and $f\in L^2(\Omega)$:
\begin{equation}
(U(t)f)(x)=f(x+t),\mbox{ for almost every }x\in \Omega\cap(\Omega-t).
\label{eq3.1.1}
\end{equation}
Moreover, given the unitary group $U$, the spectrum $\Lambda$ of the self-adjoint infinitesimal generator $\Ds$ of the group $U(t)=e^{2\pi i t\Ds}$ (as in Stone's theorem), is a spectrum for $\Omega$. 
\end{theorem}

\begin{remark}
As shown in \cite{DJ12}, and actually even in the motivation of Fuglede \cite{Fug74} for his studies of spectral sets, the self-adjoint operator $\Ds$ appearing in Theorem \ref{th3.1} are self-adjoint extensions of the differential operator $\frac{1}{2\pi i }\frac d{dx}$ on $L^2(\Omega)$. 
\end{remark}

\begin{example}
The simplest example of a spectral set is $\Omega=[0,1]$ with spectrum $\Lambda=\bz$. In this case, the group of local translations is
$$(U(t)f)(x)=f((x+t)\mod 1),\quad (f\in L^2[0,1],x,t\in[0,1]).$$
This can be checked by verifying that $U(t)e_n=e_n(t)e_n$, $t\in\br$, $n\in\bz$. 
\end{example}
\begin{proposition}\label{pr1.4}
Let $\Omega=\cup_{i=1}^n(\alpha_i,\beta_i)$ be a spectral set with spectrum $\Lambda$. Let $E$ be a Lebesgue measurable subset of $\Omega$ and $t_0\in\br$  such that $E+t_0\subset \Omega$. Then 
\begin{equation}
U(-t_0)\chi_E=\chi_{E+t_0}.
\label{eq1.4.1}
\end{equation}
\end{proposition}

\begin{proof}
Since $E+t_0$ is contained in $\Omega\cap(\Omega+t_0)$, by Theorem \ref{th3.1}, we have that, for almost every $x\in E+t_0$, $(U(-t_0)\chi_E)(x)=\chi_E(x-t_0)=\chi_{E+t_0}(x)$. Thus, if $g:=U(-t_0)\chi_E$, then 
$g(x)=\chi_{E+t_0}(x)$ for a.e., $x\in E+t_0$. On the other hand, since $U(-t_0)$ is unitary, we have that 
$\|g\|_{L^2}^2=\|\chi_E\|_{L^2}=\mu(E)$. But 
$$\mu(E)=\|g\|_{L^2}^2=\int_{E+t_0}|g(x)|^2\,dx+\int_{\Omega\setminus(\Omega+t_0)}|g(x)|^2\,dx=\mu(E+t_0)+\int_{\Omega\setminus(\Omega+t_0)}|g(x)|^2\,dx,$$
so $g(x)=0$ for a.e. $x\in \Omega\setminus(E+t_0)$.

\end{proof}

Next, we focus on spectral subsets of $\bz$ and define the one-parameter unitary group of local translations in an analogous way. As we will see, in this case, the parameter can be restricted from $\br$ to $\bz$ and thus the unitary group of local translations is determined by a local translation unitary matrix. 
\begin{definition}\label{def2.4}
Let $A$ be a finite subset of $\bz$. A {\it group of local translations} on $A$ is a continuous one-parameter unitary group $U(t)$, $t\in\br$ on $l^2(A)$ with the property that 
\begin{equation}
U(a-a')\delta_a=\delta_{a'},\quad( a,a'\in A)
\label{eq2.4.1}
\end{equation}
A unitary matrix $B$ on $l^2(A)$ is called a {\it local translation matrix} if 
\begin{equation}
B^{a-a'}\delta_a=\delta_{a'},\quad(a,a'\in A)
\label{eq2.4.2}
\end{equation}
If $A$ is a spectral subset of $\bz$ with spectrum $\Gamma$ and $|A|=N$, we define the Fourier transform from $l^2(A)$ to $l^2(\Gamma)$ by the matrix:
\begin{equation}
\F=\frac{1}{\sqrt{N}}\left(e^{-2\pi i \lambda a}\right)_{\lambda\in\Gamma,a\in A}.
\label{eq2.4.3}
\end{equation}
Let $D_\Gamma(t)$ be the diagonal matrix with entries $e^{2\pi i \lambda t}$, $\lambda\in\Gamma$. We define the group of local translations on $A$ associated to $\Gamma$ by 
\begin{equation}
U_\Gamma(t):=\F^*D_\Gamma(t)\F,\quad(t\in\br)
\label{eq2.4.4}
\end{equation}
The {\it local translation matrix} associated to $\Gamma$ is $B=U_\Gamma(1)$. 
\end{definition}

\begin{proposition}\label{pr4.4}
With the notations as in Definition \ref{def2.4}, 
\begin{equation}
U_\Gamma(t)e_\lambda=e_\lambda(t)e_\lambda,\quad Be_\lambda=e^{2\pi i\lambda}e_\lambda,\quad(\lambda\in\Gamma)
\label{eq4.4.1}
\end{equation}
Thus the vectors $e_\lambda$ in $l^2(A)$ are the eigenvectors of $B$ corresponding to the eigenvalues $e^{2\pi i \lambda}$ of multiplicity one. 

The matrix entries of $U_\Gamma(t)$ are 
\begin{equation}
U_\Gamma(t)_{aa'}=\frac1N\sum_{\lambda\in\Gamma}e^{2\pi i(a-a'+t)\lambda},\quad B_{aa'}=\frac1N\sum_{\lambda\in\Gamma}e^{2\pi i (a-a'+1)\lambda},\quad (a,a'\in A, t\in\br).
\label{eq4.4.2}
\end{equation}
\end{proposition}

\begin{proof}
We have $\F e_\lambda=\sqrt{N}\delta_\lambda$, for all $\lambda\in\Gamma$. The rest follows from an easy computation. 
\end{proof}

\begin{theorem}\label{th2.5}
Let $A$ be a spectral subset of $\bz$ with spectrum $\Gamma$ and let $U_\Gamma$ be the unitary group associated to $\Gamma$ as in Definition \ref{def2.4}. Then $U_\Gamma$ is a group of local translations on $\Gamma$, i.e., equation \eqref{eq2.4.1} is satisfied. Also $B:=U_\Gamma(1)$ is a local translation matrix. 
\end{theorem}

\begin{proof}
We have $\F\delta_a=(e^{-2\pi i \lambda a})_{\lambda\in\Gamma}$. Then $D_\Gamma(a-a')\F\delta_a=(e^{-2\pi i \lambda a'})_{\lambda\in\Gamma}=\F\delta_{a'}$. Hence $U_\Gamma(a-a')\delta_a=\delta_{a'}$.
\end{proof}

The converse holds also in the case of subsets of $\bz$, i.e., the existence of a group of local translations, or of a local translation matrix guarantees that $A$ is spectral.

\begin{theorem}\label{th2.6}
Let $A$ be a finite subset of $\bz$. The following statements are equivalent:
\begin{enumerate}
	\item $A$ is spectral.
	\item There exists a unitary group of local translations $U(t)$, $t\in\br$, on $A$. 
	\item There exists a local translation matrix $B$ on $A$. 
	
	The correspondence from (i) to (ii) is given by $U=U_\Gamma$ where $\Gamma$ is a spectrum for $A$. The correspondence from (ii) to (iii) is given by $B=U(1)$. The correspondence from (iii) to (i) is given by: if $\{e^{2\pi i \lambda} : \lambda\in\Gamma\}$ is the spectrum of $B$ then $\Gamma$ is a spectrum for $A$. 
\end{enumerate}

\end{theorem}

\begin{proof}
The implications (i)$\Rightarrow$(ii)$\Rightarrow$(iii) were proved above. We focus on (iii)$\Rightarrow$(i). 
Let $\{e^{2\pi i\lambda} : \lambda\in\Gamma\}$ be the spectrum of the unitary matrix $B$, the eigenvalues repeated according to multiplicity and let $\{v_\lambda : \lambda\in\Gamma\}$ be an orthonormal basis of corresponding eigenvectors. Let $P_\lambda$ be the orthogonal projection onto $v_\lambda$. Then 
$$B^m=\sum_{\lambda\in\Gamma}e^{2\pi i\lambda m}P_\lambda,\quad (m\in\bz).$$
We have, from \eqref{eq2.4.2},
$$\delta_{a'}=\sum_\lambda e^{2\pi i \lambda(a-a')}P_\lambda\delta_a$$
so
$P_{\lambda}\delta_{a'}=e^{2\pi i\lambda(a-a')}P_{\lambda}\delta_a$ for all $\lambda\in\Gamma$ which implies that $e^{2\pi ia\lambda}P_\lambda\delta_a$ does not depend on $a$, so it is equal to $c(\lambda)v_\lambda$ for some $c(\lambda)\in\bc$. Then $P_\lambda\delta_a=e^{-2\pi i a\lambda}c(\lambda)v_\lambda$ for all $\lambda\in\Gamma$ and 
$$\delta_a=\sum_\lambda c(\lambda)e^{-2\pi i\lambda a}v_\lambda,$$
so
$$\sum_\lambda c(\lambda)e^{-2\pi i \lambda a}v_\lambda(a')=\delta_{aa'},\quad (a,a'\in A)$$
Consider the matrices $S=(c(\lambda)e^{-2\pi i\lambda a})_{a\in A,\lambda\in\Gamma}$ and $T=(v_\lambda(a))_{\lambda\in \Gamma,a\in A}$. The previous equation implies that $ST=I$ and since $T$ is unitary, we get that $S$ is also. But then the columns have unit norm so 
$$1=\sum_{a\in A} |c(\lambda)|^2$$
and this implies that $|c(\lambda)|=\frac1{\sqrt N}$. The fact that the rows are orthonormal means that 
$$\frac1 N\sum_{\lambda\in\Gamma}e^{2\pi i(a-a')\lambda}=\delta_{aa'}.$$
But this means, first, that all the $\lambda$'s are distinct and that $\Gamma$ is a spectrum for $A$. 
\end{proof}

\begin{remark}
Given the group of local translations $U(t)$, $t\in\br$, the local translation matrix is given by $B=U(1)$. Conversely, given the local translation matrix $B$, this defines $U$ on $\bz$ in a unique way $U(n)=B^n$, $n\in\bz$. However, there are many ways to interpolate this to obtain a local translation group depending on the real parameter $t$. One can pick some choices for $\Gamma$ such that $\{e^{2\pi i \lambda } :\lambda\in\Gamma\}$ is the spectrum of $B$. Then consider the spectral decomposition 
$$B=\sum_{\lambda\in\Gamma}e^{2\pi i \lambda} P_\lambda.$$
Define 
$$U(t)=\sum_{\lambda\in\Gamma} e^{2\pi i \lambda t} P_\lambda,\quad(t\in\br).$$
Note that $U(t)$ depends on the choice of $\Gamma$. Any two such choices $\Gamma$, $\Gamma'$ are congruent modulo $\bz$, and therefore the corresponding groups $U_\Gamma(t)$ and $U_{\Gamma'}(t)$ coincide for $t\in\bz$. 

\end{remark}

\begin{example}
The simplest example of a spectral set in $\bz$ is $A=\{0,1,\dots,N-1\}$ with spectrum $\Gamma=\{0,\frac1N,\dots,\frac{N-1}N\}$. The local translation matrix associated to $\Gamma$ is the permutation matrix:
$$B=\begin{pmatrix}
0&1&\dots&0\\
0&0&\ddots&0\\
\vdots&\vdots&\vdots&\vdots\\
0&0&\dots&1\\
1&0&\dots&0
\end{pmatrix}.
$$
To see this, it is enough to check that, for $k=0,\dots,N-1$, $$B\begin{pmatrix}e_{\frac kN}(0)\\ \vdots\\e_{\frac kN}(N-1)\end{pmatrix}=e^{2\pi i \frac kN}\begin{pmatrix}e_{\frac kN}(0)\\ \vdots\\e_{\frac kN}(N-1)\end{pmatrix}=\begin{pmatrix}e_{\frac kN}(1)\\e_{\frac kN}(2)\\ \vdots\\e_{\frac kN}(N-1)\\e_{\frac kN}(0)\end{pmatrix}.$$
\end{example}

In the next proposition we link the two concepts for $\bz$ and for $\br$: if $A$ is a spectral set in $\bz$, with spectrum $\Gamma$, then $\Omega=A+[0,1]$ is a spectral set in $\br$ with spectrum $\Lambda=\Gamma+\bz$. The local group of local translations for $\Omega$ and $\Lambda$ can be expressed in terms of the local translation matrix associated to $A$ and $\Gamma$. 

\begin{proposition}\label{pr1.5}
Let $A=\{a_0,\dots,a_{N-1}\}$ be a spectral set in $\bz$, with spectrum $\Gamma=\{\lambda_0,\dots,\lambda_{N-1}\}$. Then the set $\Omega=A+[0,1]$ is spectral in $\br$ with spectrum $\Lambda=\Gamma+[0,1]$. 
Define the matrix of the Fourier transform from $l^2(A)$ to $l^2(\Gamma)$: 
\begin{equation}
\F=\F_{A,\Gamma}=\frac{1}{\sqrt{N}}\left(e^{-2\pi i\lambda_j a_k}\right)_{j,k=0}^{N-1},
\label{eq1.5.1}
\end{equation}
and let $D_\Gamma$ be the $N\times N$ diagonal matrix with entries $e^{2\pi i\lambda_j}$, $j=0,\dots,N-1$.
\begin{equation}
B=\F^* D_\Gamma\F
\label{eq1.5.2}
\end{equation}
The group of local translations $(U_\Lambda(t))_{t\in\br}$ associated to the spectrum $\Lambda$ of $\Omega$ is given by
\begin{equation}
\begin{pmatrix}
(U_\Lambda(t)f)(x+a_0)\\
\vdots\\
(U_\Lambda(t)f)(x+a_{N-1})
\end{pmatrix}=B^{\lfloor x+t\rfloor}\begin{pmatrix}
f(\{ x+t\}+a_0)\\ \vdots\\
f(\{x+t\}+a_{N-1})
\end{pmatrix},\quad (f\in L^2(\Omega),x\in[0,1],t\in\br).
\label{eq1.5.3}
\end{equation}
where $\lfloor\cdot\rfloor$ and $\{\cdot\}$ represent the integer and the fractional parts respectively.
\end{proposition}

\begin{proof}
The fact that $\Lambda$ is a spectrum for $\Omega$ can be found, for example, in \cite{KoMa06}. The formula for $U_\Lambda$ appears in a slightly different form in \cite{DJ12u}, but we can check it here directly in a different way: it is enough to prove that 
\begin{equation}
U_\Lambda(t)e_\lambda=e_\lambda(t)e_\lambda\mbox{ for all }\lambda\in\Lambda
\label{eq1.5.4}
\end{equation}
Thus, we have to plug in $f=e_{\lambda_i+n}$ in the right hand side of \eqref{eq1.5.3}, with $\lambda_i\in\Gamma$, $n\in\bz$. For the computation, we will use the following relation:

\begin{equation}
\F\frac{1}{\sqrt N}\begin{pmatrix}
e^{2\pi i \lambda_i a_0}\\
\vdots\\
e^{2\pi i \lambda_i a_{N-1}}
\end{pmatrix}=\delta_i.
\label{eq1.5.5}
\end{equation}
Indeed, we have, for $j=0,\dots,N-1$,
$$\frac{1}{N}\sum_{k=0}^{N-1}e^{-2\pi i\lambda_j a_k}e^{2\pi i\lambda_i a_k}=\frac{1}{N}\sum_{k=0}^{N-1}e^{2\pi i(\lambda_i-\lambda_j)a_k}=\delta_{ij},$$
because $\Gamma$ is a spectrum for $A$.

Then, for $m\in\bz$,
\begin{equation}
B^m\begin{pmatrix}
e^{2\pi i \lambda_i a_0}\\
\vdots\\
e^{2\pi i \lambda_i a_{N-1}}
\end{pmatrix}=e^{2\pi i\lambda_i m}\begin{pmatrix}
e^{2\pi i \lambda_i a_0}\\
\vdots\\
e^{2\pi i \lambda_i a_{N-1}}
\end{pmatrix}
\label{eq1.5.6}
\end{equation}
We have, for $x\in[0,1]$ and $t\in\br$:
$$B^{\lfloor x+t\rfloor}\begin{pmatrix}
e^{2\pi i(\{x+t\}+a_0)(\lambda_i+n)}\\
\vdots\\
e^{2\pi i(\{x+t\}+a_{N-1})(\lambda_i+n)}
\end{pmatrix}
=e^{2\pi i \{x+t\}(\lambda_i+n)}B^{\lfloor x+t\rfloor}\begin{pmatrix}
e^{2\pi ia_0\lambda_i}\\
\vdots\\
e^{2\pi ia_{N-1}\lambda_i}
\end{pmatrix}$$$$
=e^{2\pi i \{x+t\}(\lambda_i+n)}e^{2\pi i\lfloor x+t\rfloor\lambda_i}\begin{pmatrix}
e^{2\pi ia_0\lambda_i}\\
\vdots\\
e^{2\pi ia_{N-1}\lambda_i}
\end{pmatrix}$$$$=e^{2\pi i(\lambda_i+n)(x+t)}\begin{pmatrix}
e^{2\pi ia_0\lambda_i}\\
\vdots\\
e^{2\pi ia_{N-1}\lambda_i}
\end{pmatrix}=e^{2\pi i(\lambda_i+n) t}\begin{pmatrix}
e^{2\pi i(x+a_0)(\lambda_i+n)}\\
\vdots\\
e^{2\pi i(x+a_{N-1})(\lambda_i+n)}
\end{pmatrix}.$$
This proves \eqref{eq1.5.4}.
\end{proof}

In the following we present some connections between the local matrix $B$ and possible tilings for the set $A$. We define a set $\Theta_B$ as the set of powers of the matrix $B$ which have a canonical vector as a column, with 1 not on the diagonal. 
\begin{definition}\label{def3.8}
Let $A$ be a spectral subset of $\bz$ with spectrum $\Gamma$. Let $B$ be the associated local translation matrix. Define
\begin{equation}
\Theta_B:=\{ m\in\bz : B^m\mbox{ has a column $a$ equal to the canonical vector $\delta_{a'}$ for some $a\neq a'$}\}
\label{eq3.8.1}
\end{equation}
$$=\{m\in\bz : B^m\delta_a=\delta_{a'}\mbox{ for some $a\neq a'$ in $A$}\}.$$
\end{definition}

\begin{proposition}\label{pr3.9}
Let $A$ be a spectral subset of $\bz$ with spectrum $\Gamma$, $|A|=N$. Assume $0\in\Gamma$. Assume in addition that the smallest lattice that contains $\Gamma$ is $\frac rd\bz$ for some mutually prime integers $r,d\geq 1$. For a subset $\mathcal T$ of $\bz$ the following statements are equivalent:
\begin{enumerate}
	\item $\mathcal T\oplus A=\bz_d$, in the sense that $\T\oplus A$ is a complete set of representatives modulo $d$ and every element $x$ in $\T+A$ can be represented in a unique way as $x=t+a$ with $t\in\T$ and $a\in A$. In this case $A$ tiles $\bz$ by $\T\oplus d\bz$. 
	\item $(\T-\T)\cap \Theta_B=\{0\}$ and $|\T||A|=d$. 
\end{enumerate}

\end{proposition}
\begin{proof}
First, we present $\Theta_B$ in a more explicit form. By Proposition \ref{pr4.4} we have 
$$(B^m)_{aa'}=\frac{1}{N}\sum_{\lambda\in\Gamma}e^{2\pi i (a-a'+m)\lambda}.$$
Since we want $(B^m)_{aa'}$ to be 1 for some $a\neq a'$, we must have equality in the triangle inequality 
$$|(B^m)_{aa'}|\leq \frac{1}{N}\sum_{\lambda\in\Gamma}1=1,$$
and since $0\in\Gamma$ this implies that $e^{2\pi i (a-a'+m)\lambda}=1$ so $(a-a'+m)\lambda\in\bz$ for all $\lambda\in\Gamma$. Since the smallest lattice that contains $\Gamma$ is $\frac rd\bz$, we obtain that $(a-a'+m)\frac rd\in\bz$ which means that $m\equiv a'-a \mod d$. The converse also holds: if $a'\equiv a+m\mod d$ then $B^m$ has a 1 on position $aa'$. Thus 
\begin{equation}
\Theta_B=\{m\in\bz : m\equiv a'-a\mod d\mbox{ for some }a\neq a'\in A\}
\label{eq3.9.1}
\end{equation}

(i)$\Rightarrow$(ii). Suppose there exists $t\neq t'$ in $\T$ such that $t-t'\in \Theta_B$. Then there exist $a\neq a'$ in $A$ such that $a'-a\equiv t-t'\mod d$. Then $a'+t'\equiv a+t\mod d$, a contradiction. Also if $A\oplus \T=\bz_d$, then $|A||\T|=d$. 

(ii)$\Rightarrow$(i). It is enough to prove that $(A-A)\cap(\T-\T)=\{0\}\mod d$, because this implies that the map from $A\times T$ to $(A+T)\mod d$, $(a,t)\mapsto a+t\mod d$ is injective, and the condition $|A||\T|=d$ implies that it has to be bijective. 

Suppose not. Then there exist $a\neq a'$ in $A$ $t\neq t'$ in $\T$ such that $a+t\equiv a'+t'\mod d$. Then $t-t'\equiv a'-a\mod d$. Therefore $t-t'\in \Theta_B$ which contradicts the hypothesis. 
\end{proof}

\begin{corollary}\label{cor3.10}
If the local translation matrix $B$ is 
$$B=\begin{pmatrix}
0&1&\dots&0\\
0&0&\ddots&0\\
\vdots&\vdots&\vdots&\vdots\\
0&0&\dots&1\\
1&0&\dots&0
\end{pmatrix},$$
then $A$ tiles $\bz$ by $N\bz$.
\end{corollary}

\begin{proof}
We have $\Theta_B=\{1,\dots,N-1\}+N\bz$ so one can take $\mathcal T=\{0\}$ in Proposition \ref{pr3.9}.
\end{proof}

Another piece of information that is contained in the local translation matrix is the rationality of the spectrum:

\begin{proposition}\label{pr3.11}
Let $A$ be a spectral set in $\bz$ with spectrum $\Gamma$ and local translation matrix $B$. Let $d\in\bz$, $d\geq1$. Then $\Gamma\subset\frac1d\bz$ if and only if $B^d=I$. The spectrum $\Gamma$ is rational if and only if the group of local translations $U_\Gamma$ has an integer period, i.e., there exists $p\in\bz$, $p\geq 1$ such that $U_\Gamma(t+p)=U_\Gamma(t)$, $t\in\br$. 
\end{proposition}
\begin{proof}
If $\Gamma\subset \frac1d\bz$, using equation \eqref{eq2.4.4}, with $t=d$, we have that $D_\Gamma(d)=I$ so $B^d=I$. Conversely, if $B^d=I$ then $D_\Gamma(d)=I$ and therefore $d\lambda\in\bz$ for all $\lambda\in\Gamma$. The second statement follows from the first. 
\end{proof}

\section{Examples}
In this section we study the local translation groups associated to spectral sets of low cardinality $N=2,3,4,5$. Such sets were described in \cite{DH12}. We recall here the results:
\begin{definition}\label{def5.1}
The {\it standard $N\times N$ Hadamard matrix} is 
\begin{equation}
\frac{1}{\sqrt{N}}\left(e^{2\pi i\frac{jk}N}\right)_{j,k=0}^{N-1}.
\label{eq5.1}
\end{equation}
We say that a $N\times N$ matrix is {\it equivalent to the standard Hadamard matrix} if it can be obtained from it by permutations of rows and columns. 

Let $A$ and $L$ be two subsets of $\bz$ and $R\in\bz$, $R\geq1$. We say that $(A,L)$ is a {\it Hadamard pair} with scaling factor $R$ if $\frac1RL$ is a spectrum for $A$.
\end{definition}

\begin{theorem}\label{th5.1}
Let $A \subset \bz$ have $N$ elements and spectrum $\Gamma$. Assume $0$ is in $A$ and $\Gamma$. Suppose the Hadamard matrix associated to $(A,\Gamma)$ is equivalent to the standard $N$ by $N$ Hadamard matrix. Then $A$ has the form $A=d A_0$ where $d$ is an integer and $A_0$ is a complete set of residues modulo $N$ with $\gcd(A_0)=1$. In this case any such spectrum $\Gamma$ has the form $\Gamma = \frac{1}{R} f L_0$ where $f$ and $R$ are integers, $L_0$ is a complete set of residues modulo $N$ with greatest common divisor one, and $R=NS$ where $S$ divides $df$ and $\frac{df}{S}$ is mutually prime with $N$. The converse also holds.

\end{theorem}

Since for $N=2,3,5$ our Hadamard matrices are equivalent to the standard one (see \cite{Haa97,TaZy06}) the next corollary follows:
\begin{corollary}\label{pr5.1}A set $A \subset \bz$ with $|A|=N=2$, $3$, or $5$, where $0 \in A$ is spectral if and only if $A= N^k A_0$ where $k$ is a positive integer and $A_0$ is a complete set of residues modulo $N$.
\end{corollary}

For cardinality $N=4$ the situation is more complex:
\begin{theorem}\label{thha4}
Let $A$ be spectral with spectrum $\Gamma$ and size $N=4$. Assume $0$ is in both sets. Then there exists a set of integers $L$, containing $0$, and an integer scaling factor $R$ so that $\Gamma= \frac{1}{R} L$.

$(A,L)$ is a Hadamard pair (each containing $0$) of integers of size $N=4$, with scaling factor $R$, if and only if $R=2^{C+M+a+1} d$, $A=2^C \{0, 2^a c_1, c_2, c_2 + 2^a c_3\}$, and $L=2^M \{0, n_1, n_1 + 2^a n_2, 2^a n_3\}$, where $c_i$ and $n_i$ are all odd, $a$ is a positive integer, $C$ and $M$ are non-negative integers, and $d$ divides $c_1 n$, $c_3 n$, $n_2 c$, and $n_3 c$, where $c$ is the greatest common divisor of the $c_k$'s and similarly for $n$.
\end{theorem}

The next proposition helps us simplify our study:

\begin{proposition}\label{pr5.5}
Let $A$ be a spectral set in $\bz$ with spectrum $\Gamma$, local translation group $U_\Gamma$ and local translation matrix $B$. Let $d\in\bz$ $d\geq 1$. Then $dA$ is spectral with spectrum $\frac1d\Gamma$. The local translation group $U_{\frac1d\Gamma}$ and the local translation matrix $B_{\frac1d\Gamma}$ are related to the corresponding ones for $A$ and $\Gamma$ by
\begin{equation}
U_{\frac1d\Gamma}(t)=U_\Gamma(\frac td),\quad (t\in\br),\quad B_{\frac1d\Gamma}^d=B_\Gamma
\label{eq5.5.1}
\end{equation}
\end{proposition}

\begin{proof}
Everything follows from \eqref{eq2.4.4} by a simple calculation.
\end{proof}

\medskip

{\bf N=2.} We can take $A=2^c\{0,a_0\}$, with $c\in\bz$, $c\geq0$ and $a_0$ odd, and $\Gamma=\{0,\gamma_1=\frac g{2^{c+1}}\}$ with $g$ odd. The matrix of the Fourier transform is 
$$\F=\frac{1}{\sqrt2}\begin{pmatrix} 1&1\\1&-1\end{pmatrix}.$$
By equation \eqref{eq2.4.4}, we can compute the local translation matrix $B$ and the local translation group $U_\Gamma$:
\beq 
B =\frac12 \begin{pmatrix}
  1&1\\
  1&-1
\end{pmatrix}
\begin{pmatrix}
  1&0\\
  0&e^{ \pi i  \gamma_1 }
\end{pmatrix}
\begin{pmatrix}
  1&1\\
  1& -1
\end{pmatrix} .
\eeq
Here $\gamma_1$ depends on the non-zero element of $A$, called $a_1$. Let $a_1 = 2^c a_0$, where $a_0$ is odd. Then $\gamma_1 = \frac{g}{2^{c+1}}$, where $g$ is odd.
Multiplying, we obtain
\beq 
B = \frac12 \begin{pmatrix}
  1+e^{ \pi i  \gamma_1 }&1-e^{ \pi i  \gamma_1 }\\
  1-e^{ \pi i  \gamma_1 }&1+e^{ \pi i  \gamma_1 }
\end{pmatrix} .
\eeq
We also have
\beq 
U_\Gamma(t) = \frac12 \begin{pmatrix}
  1+e^{t \pi i  \gamma_1}&1-e^{t \pi i  \gamma_1}\\
  1-e^{t \pi i  \gamma_1}&1+e^{t \pi i  \gamma_1}
\end{pmatrix} .
\eeq 

Note that when $c=0$,
$$B=\begin{pmatrix} 0&1\\1&0\end{pmatrix}.$$
In this case $\Theta_B=\bz$. 
\medskip

{\bf N=3.} We can take $A=3^c\{0,3j_1+1,3j_2+2\}$ and $\Gamma=\{0,\gamma_1=\frac{3j_3+1}{3^{c+1}},\gamma_2=\frac{3j_4+2}{3^{c+1}}\}$. The matrix of the Fourier transform is, with $\omega=e^{\frac{2\pi i}{3}}$:
$$\F=\frac{1}{\sqrt3}\begin{pmatrix} 1&1&1\\
1&\cj\omega&\omega\\
1&\omega&\cj\omega\end{pmatrix}.$$
We compute the group of local translations:
\beq 
U_\Gamma(t) = \frac13
\begin{pmatrix}
  1 &1 &1\\
  1& {\omega} & \cj\omega \\
  1 & \cj\omega & {\omega}
\end{pmatrix}
\begin{pmatrix}
  1 &0 &0\\
  0& e^{2 \pi i \gamma_1 t} & 0 \\
  0 & 0 & e^{2 \pi i \gamma_2 t}
\end{pmatrix}
 \begin{pmatrix}
  1 &1 &1\\
  1& \cj\omega & {\omega} \\
  1 & {\omega} & \cj\omega
\end{pmatrix}.
\eeq
Multiplying, we obtain
{\small
$$U_\Gamma(t)=\frac13\begin{pmatrix}
1+e^{2\pi i\gamma_1t}+e^{2\pi i\gamma_2t}& 1+\cj\omega e^{2\pi i\gamma_1t}+\omega e^{2\pi i \gamma_2t}& 1+\omega e^{2\pi i\gamma_1t}+\cj\omega e^{2\pi i\gamma_2t}\\
1+\omega e^{2\pi i\gamma_1 t}+\cj\omega e^{2\pi i\gamma_2 t}& 1+e^{2\pi i\gamma_1 t}+e^{2\pi i \gamma_2 t}& 1+\cj \omega e^{2\pi i\gamma_1 t}+\omega e^{2\pi i\gamma_2 t}\\
1+\cj\omega e^{2\pi i\gamma_1 t}+\omega e^{2\pi i\gamma_2 t}& 1+ \omega e^{2\pi i\gamma_1 t}+\cj\omega e^{2\pi i \gamma_2 t}& 1+e^{2\pi i\gamma_1 t}+e^{2\pi i\gamma_2 t}
\end{pmatrix}
$$
}

Note that, when $c=0$, 
$$B=U_\Gamma(1)=\begin{pmatrix} 0&1&0\\
0&0&1\\
1&0&0
\end{pmatrix}.$$
In this case $\Theta_B=\bz$.

\medskip

{\bf N=4.} We take a simple case to obtain some nice symmetry, so we will ignore, after some rescaling, the common factor. So take $A=\{0,2^a c_1 , c_2 , c_2 + 2^a c_3 \}$, $\Gamma=\frac{1}{2^{a+1}} \{0,2^a n_1 , n_2 , n_2 + 2^a n_3 \}$ as in Theorem \ref{thha4}. The matrix of the Fourier transform is 
\beq
\frac12 \begin{pmatrix}
1& 1 & 1 & 1 \\
1& 1 & -1 & -1 \\
1& -1 & \rho & -\rho \\
1& -1 & -\rho & \rho \\
\end{pmatrix},
\eeq
where $\rho = \text{exp}\left(-\frac{ \pi i c_2 n_2}{2^a} \right)$. We compute integers powers of the spectral matrix $B$:
\beq
B^k = \frac14 \begin{pmatrix}
1& 1 & 1 & 1 \\
1& 1 & -1 & -1 \\
1& -1 & \rho & -\rho \\
1& -1 & -\rho & \rho \\
\end{pmatrix}^*
 \begin{pmatrix}
1& 0 & 0 & 0 \\
0& (-1)^k & 0 & 0 \\
0& 0 & z^k & 0 \\
0& 0 & 0 & (-z)^k \\
\end{pmatrix}
 \begin{pmatrix}
1& 1 & 1 & 1 \\
1& 1 & -1 & -1 \\
1& -1 & \rho & -\rho \\
1& -1 & -\rho & \rho \\
\end{pmatrix},
\eeq
where $z=\text{exp}\left(\frac{ \pi i n_2}{2^{a}}\right)$.

We obtain for odd $k$,
\beq
B^k =\frac12 \begin{pmatrix}
0& 0 & 1+z^k \rho & 1-z^k \rho \\
0& 0 & 1-z^k \rho & 1+z^k \rho \\
1+z^k \overline{\rho} & 1-z^k \overline{\rho} & 0 & 0 \\
1-z^k \overline{\rho} & 1+z^k \overline{\rho} & 0 & 0 \\
\end{pmatrix}.
\eeq
We obtain for even $k$,
\beq
B^k =\frac12 \begin{pmatrix}
1+z^k  & 1-z^k  & 0 & 0 \\
1-z^k  & 1+z^k  & 0 & 0 \\
0& 0 & 1+z^k  & 1-z^k  \\
0& 0 & 1-z^k  & 1+z^k  \\
\end{pmatrix}.
\eeq
We compute $\Theta_B$. We have $k\in\Theta_B$ if and only if one of the following situations occurs: $z^k=-1$, $z^k\rho=\pm 1$, $z^k\cj\rho=\pm1$. This means that 
$\frac{kn_2}{2^a}=2m+1$  or $\frac{kn_2}{2^a}\pm \frac{c_2n_2}{2^a}=m$ for some $m\in\bz$. Since $n_2$ is odd, this implies that
$$\Theta_B=\{2^a(2m+1) :  m\in\bz\}\cup\{ 2^am\pm c_2 : m\in\bz\}.$$

Then, one can easily see that $\T:=\{0,2,4,\dots, 2^a-2\}$ satisfies the conditions in Propositon \ref{pr3.9}, and therefore $A\oplus \T=\bz_{2^{a+1}}$. 
\begin{example}
Let $A=\{0,1,4,5\}$ and $\Gamma=\frac18\{0,1,4,5\}$. We illustrate how the group of local translations acts on an indicator function.

\begin{figure}[htb]
\begin{center}
\includegraphics[scale=0.2]{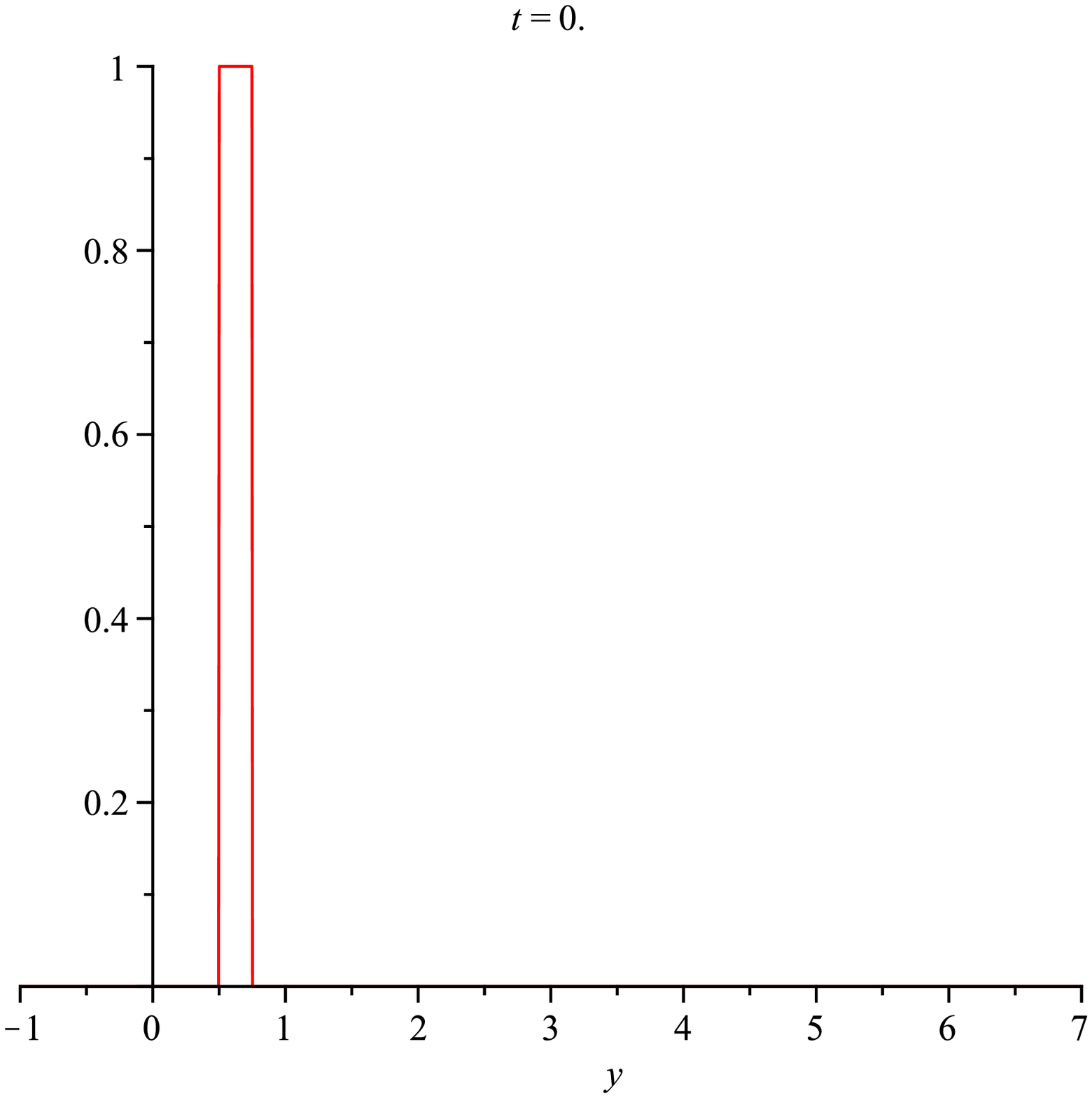}
\includegraphics[scale=0.2]{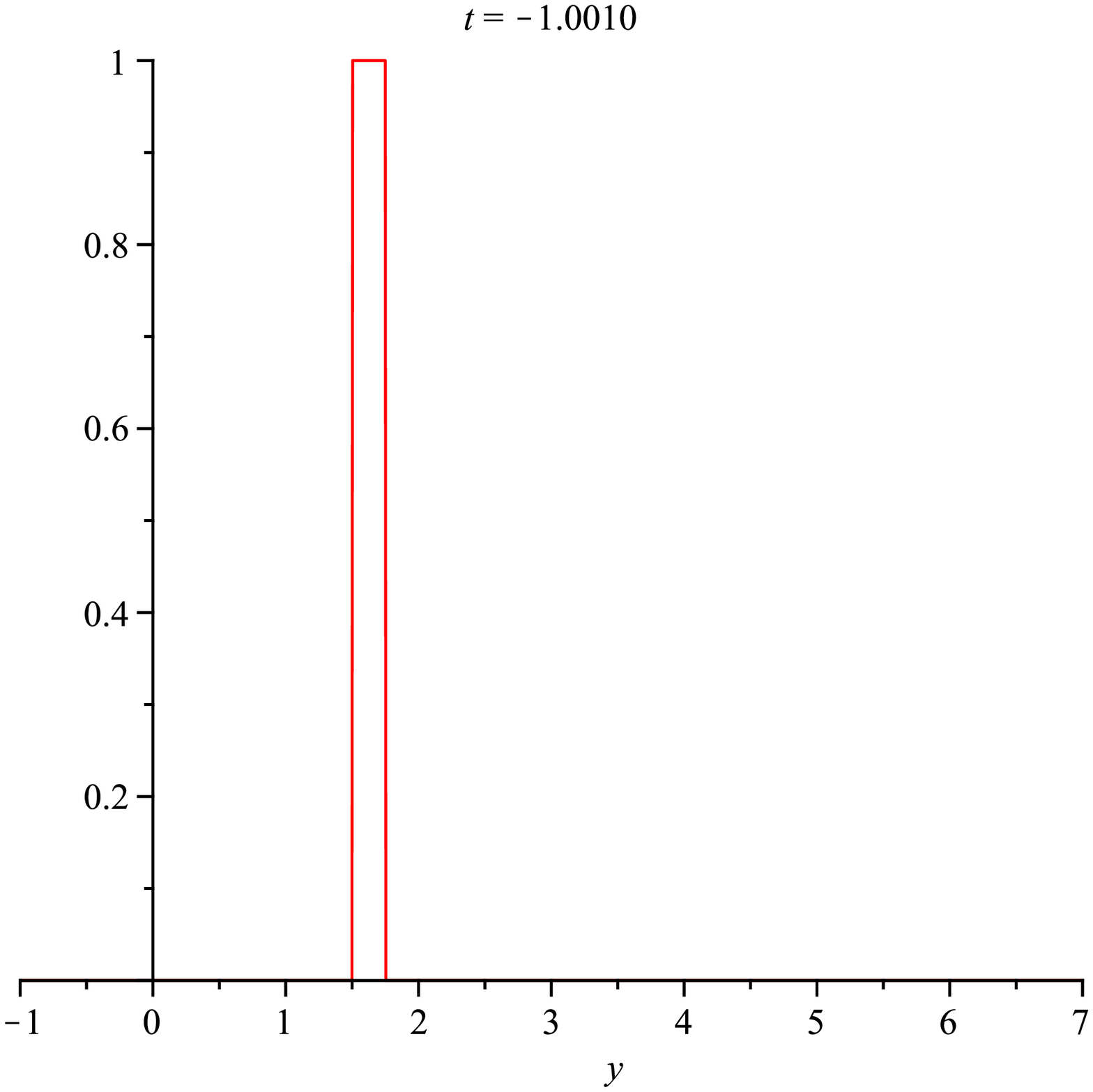}
\includegraphics[scale=0.2]{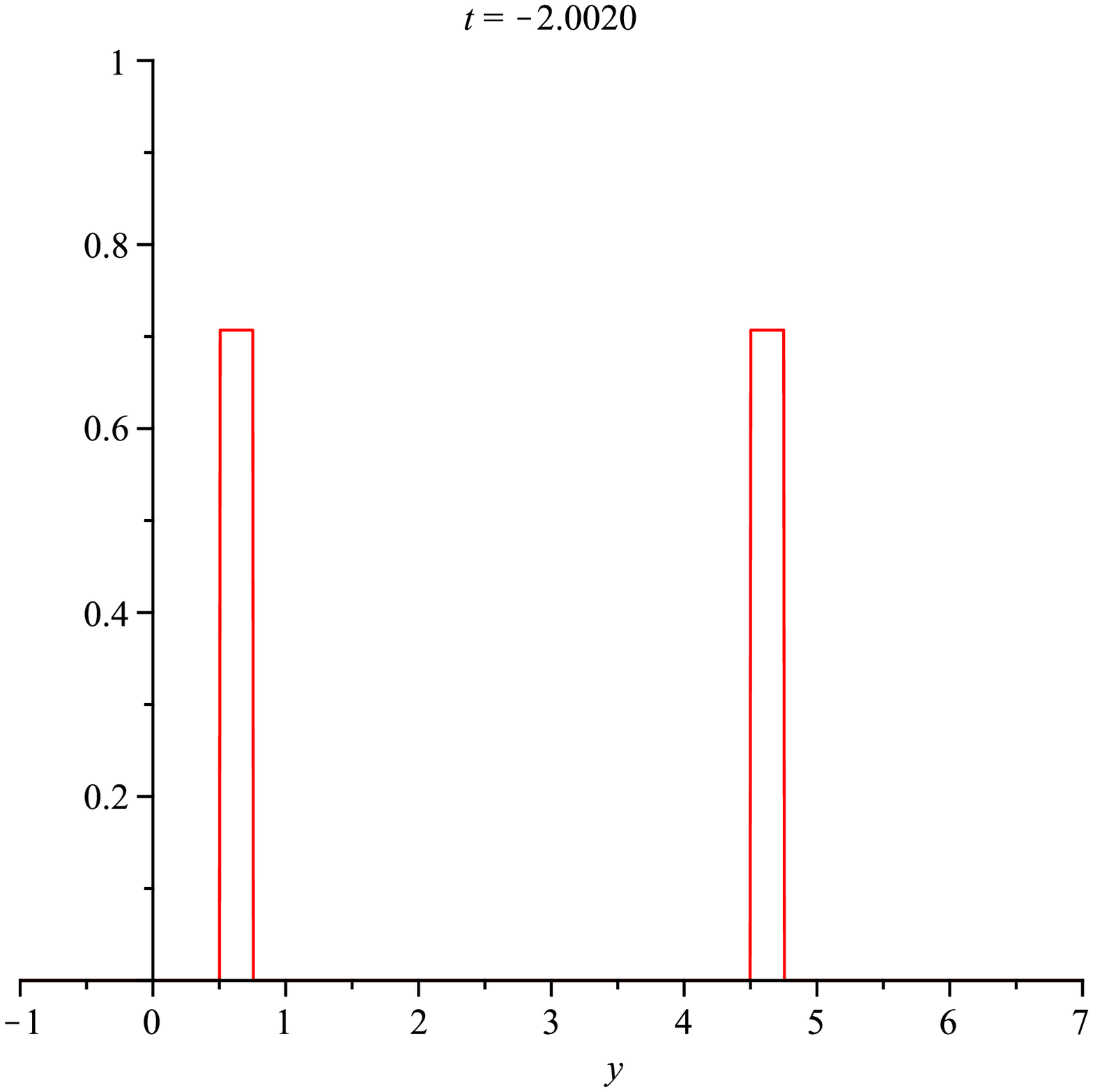}

\includegraphics[scale=0.2]{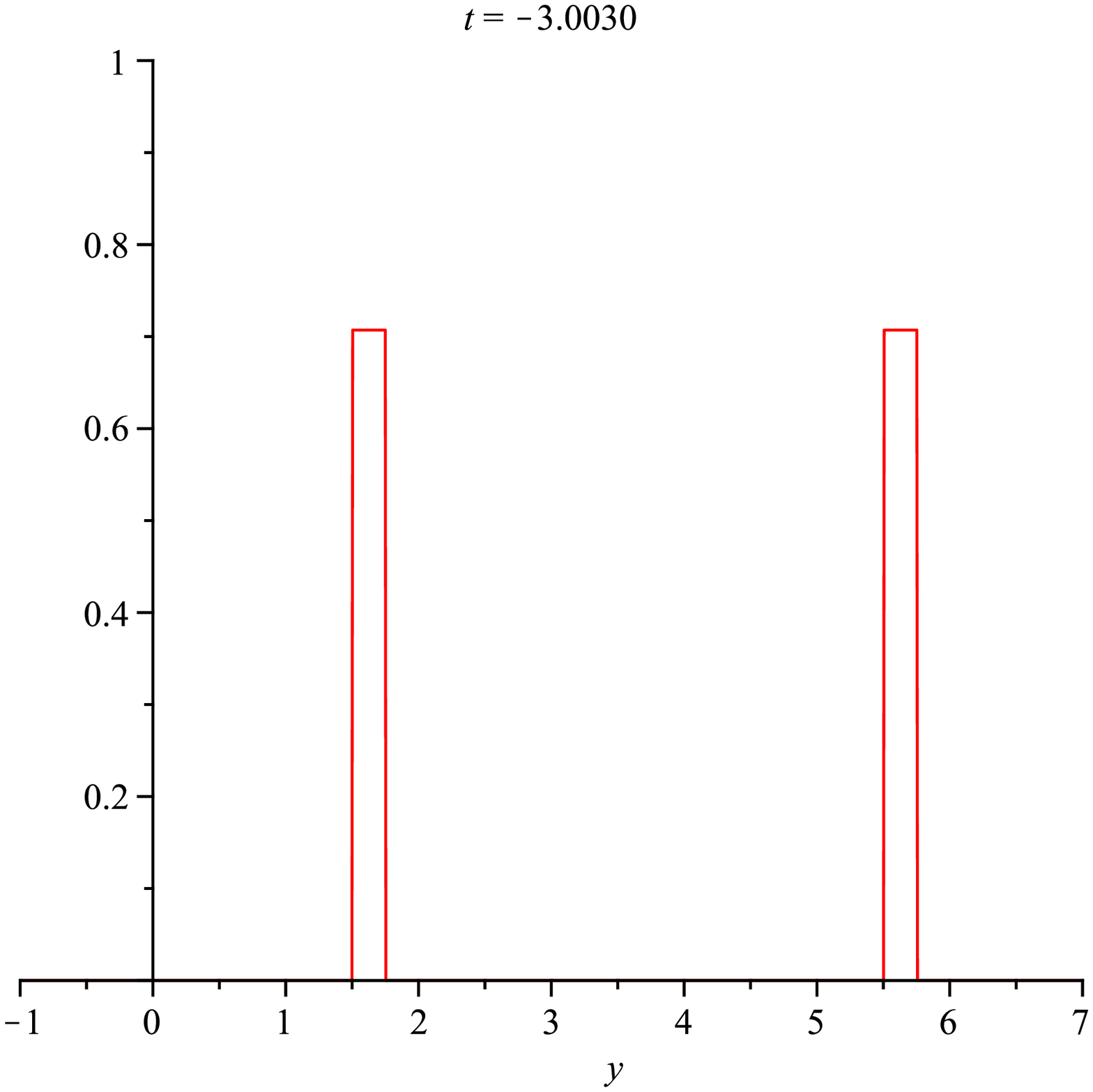}
\includegraphics[scale=0.2]{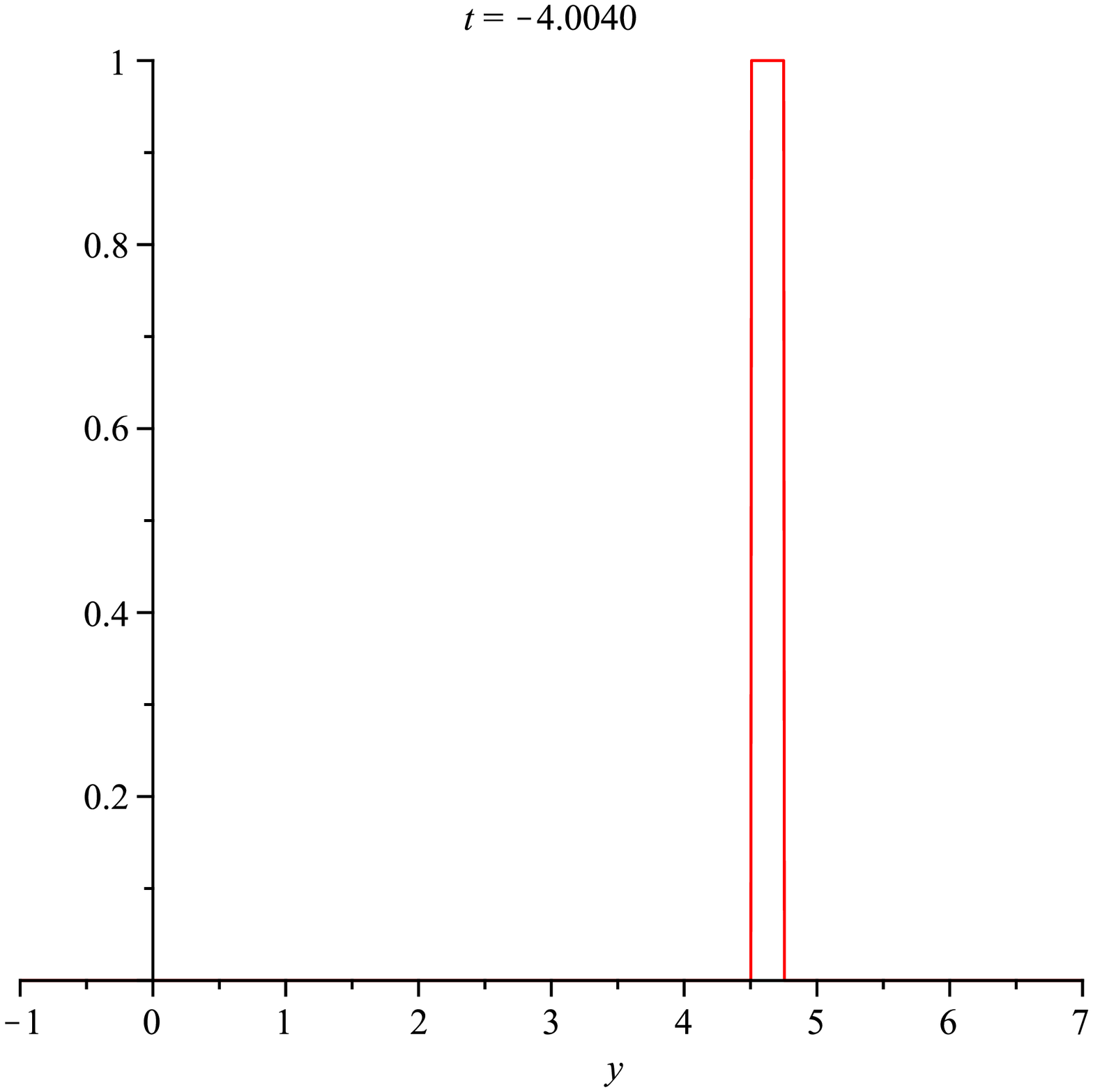}
\includegraphics[scale=0.2]{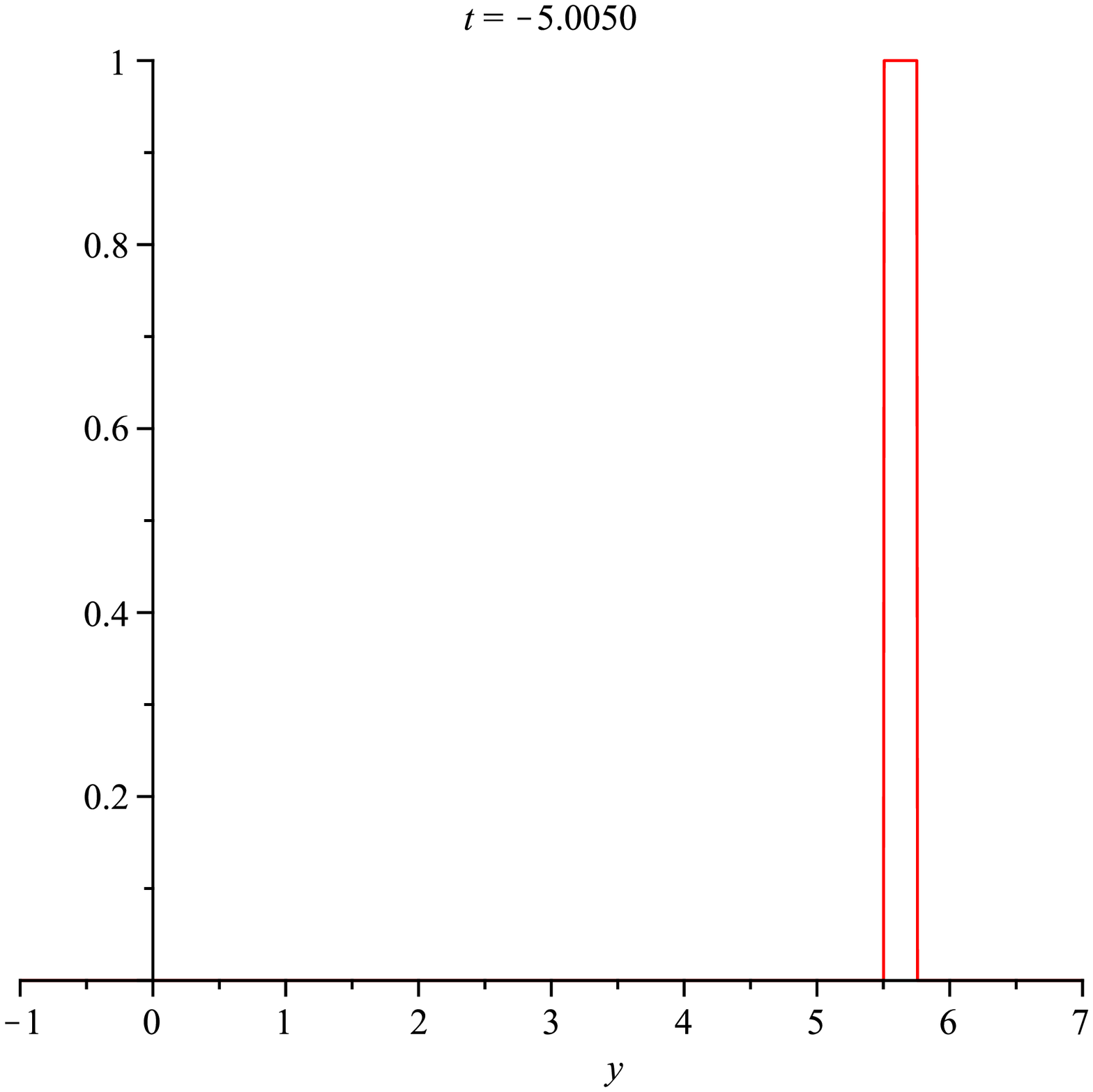}

\includegraphics[scale=0.2]{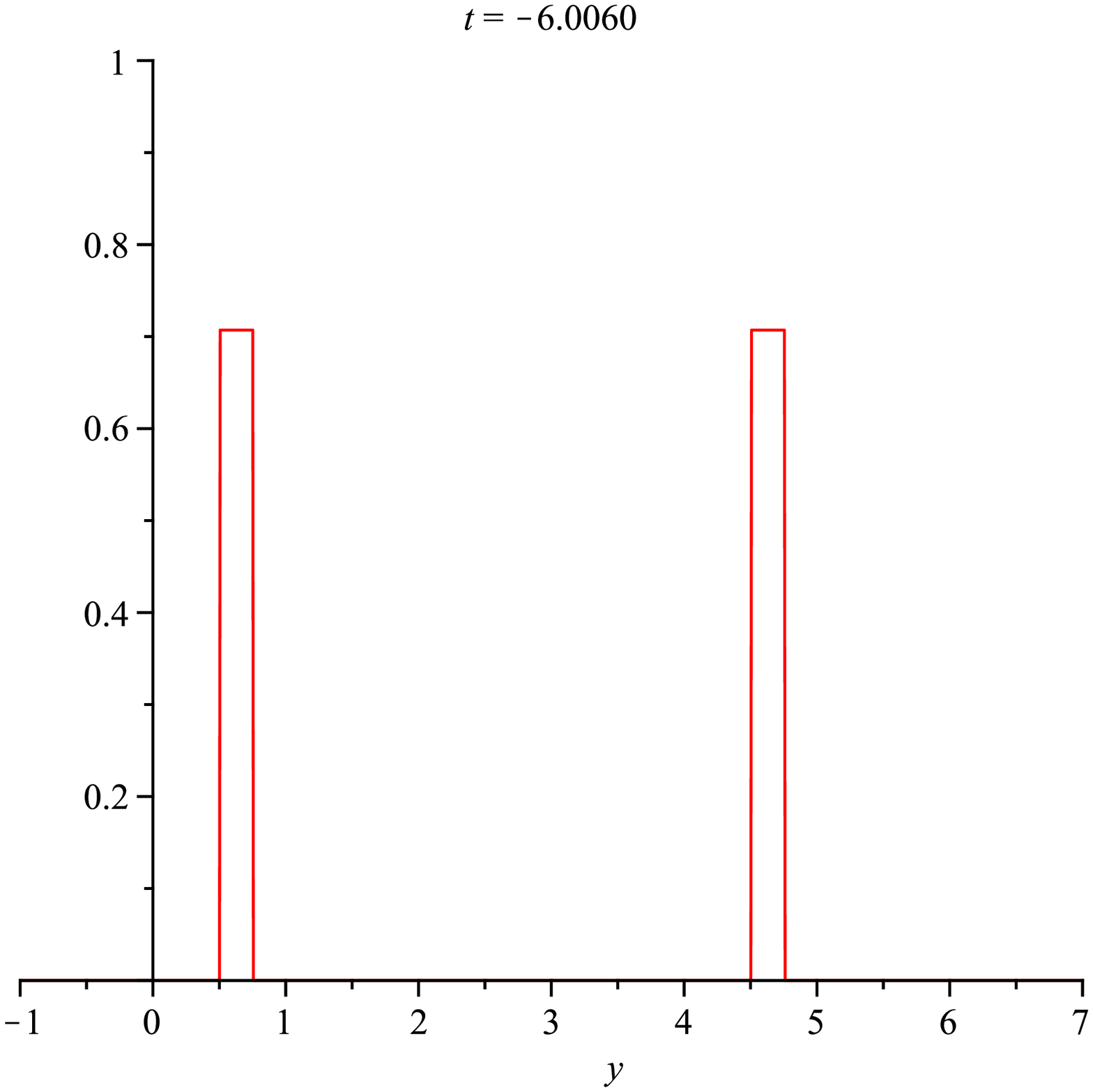}
\includegraphics[scale=0.2]{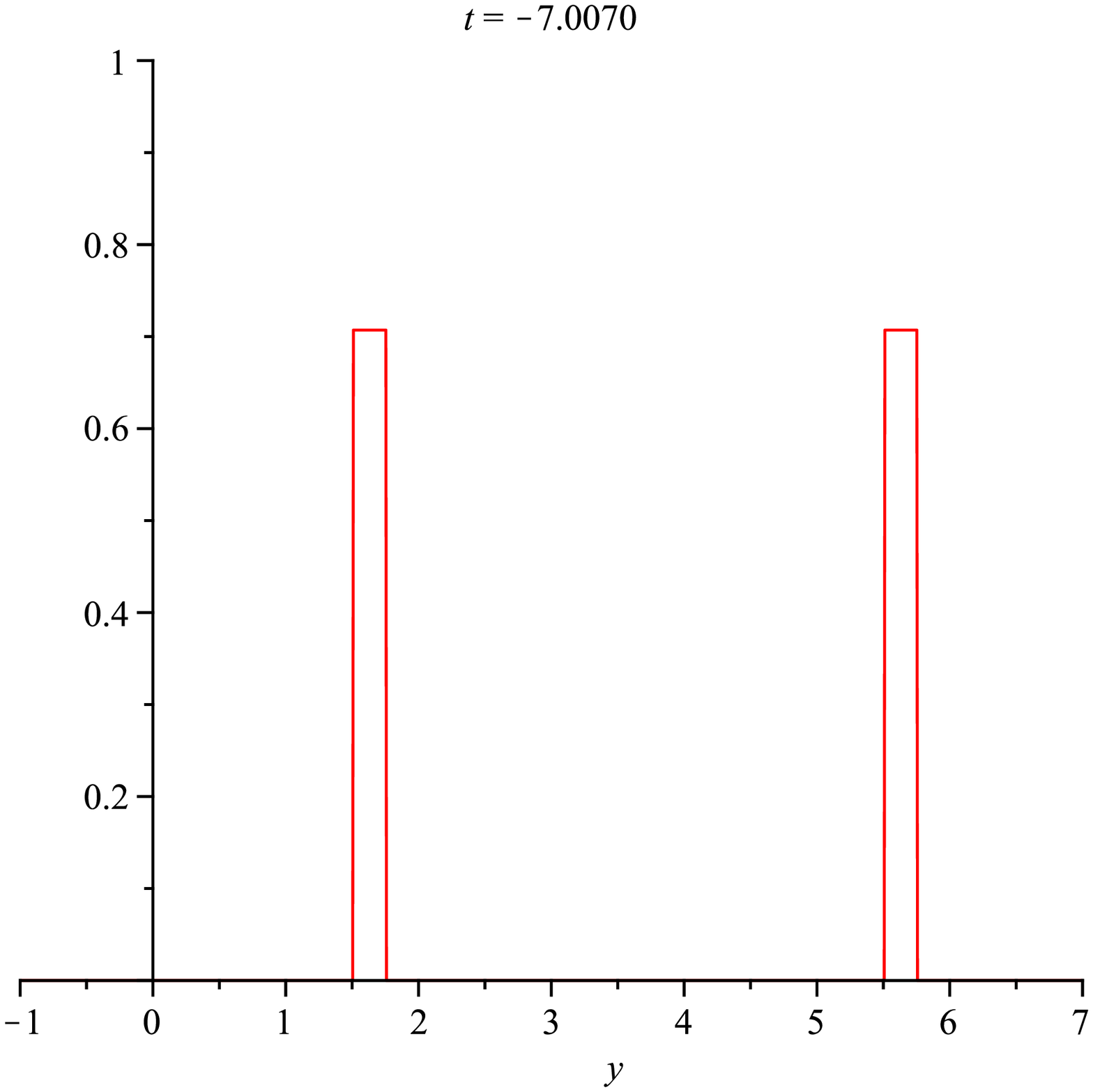}
\includegraphics[scale=0.2]{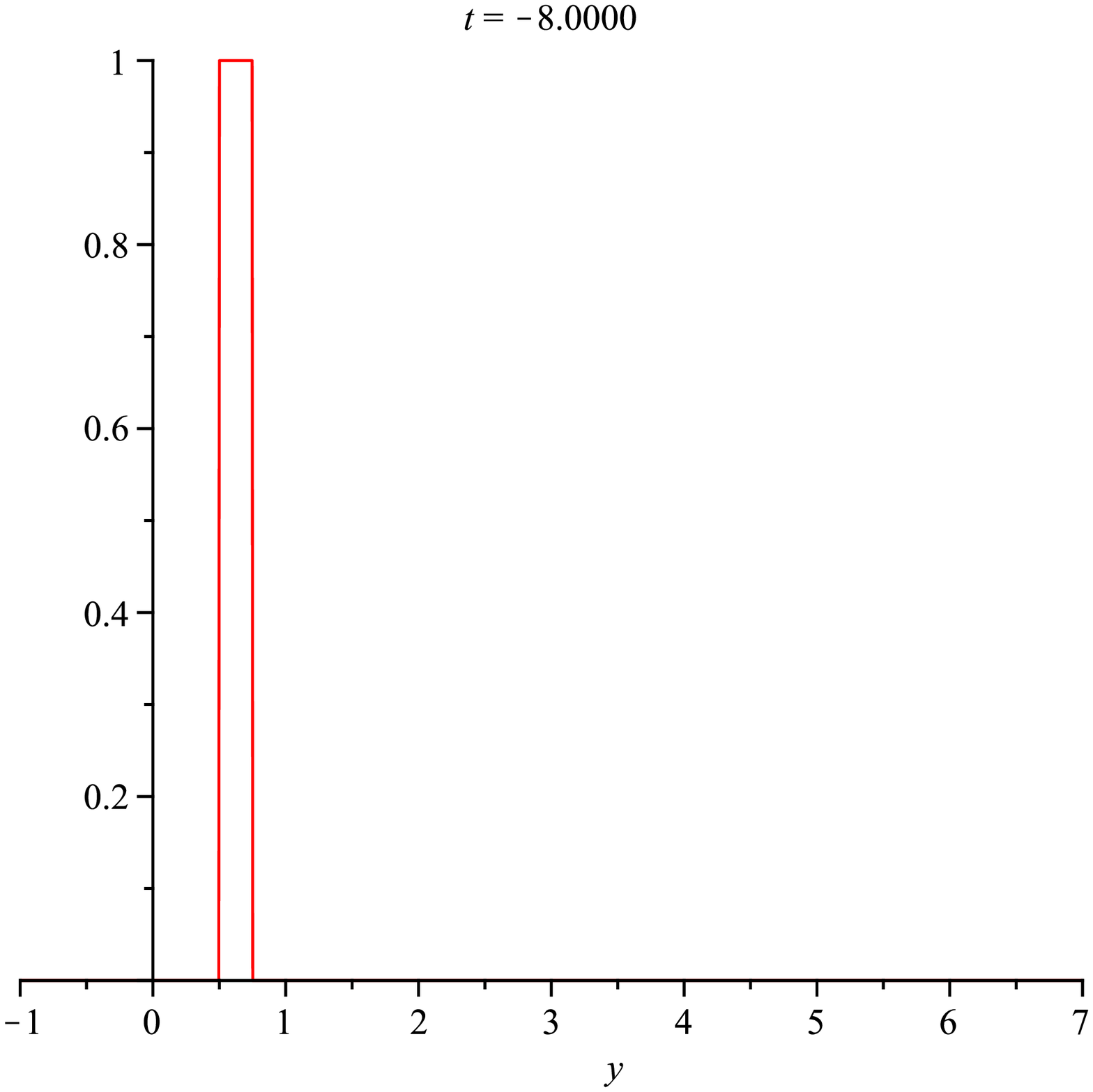}
\end{center}
\caption{Local translations for $A=\{0,1,4,5\}$ and $\Gamma=\frac18\{0,1,4,5\}$. }
\end{figure}
The indicator function $f$ is the one in the first picture, for $t=0$. We use negative values for $t$ to move the function to the right. We show here the absolute value of $U(t)f$. Note that for $t\approx -1, -4$ and $-5$, since the the interval $[0,1]$ is moved into the intervals $[1,2]$, $[4,5]$ $[5,6]$, which are contained in $A+[0,1]$, as predicted by the theory, e.g., Proposition \ref{pr1.4}, the group $U(t)$ really acts a simple translation. 

For $t\approx 2$, the interval $[0,1]+2$ is no longer contained in $A+[0,1]$. The local translation $U(-2)$ splits the indicator function into 2 pieces, supported on $[0,1]$ and $[4,5]$. Similarly for $t\approx -3,-6,-7$. Since $\Gamma$ is contained in $\frac18\bz$, the group of local translations has period 8. We see this in the last picture $U(-8)f=f$.
\end{example}
\medskip

{\bf N=5.} For simplicity, by rescaling we can ignore the common factors in $A$ and $\Gamma$ so we take $A=\{0,a_1,a_2,a_3,a_4\}$ with $a_j\equiv j\mod5$ and $\Gamma=\frac15\{0,\gamma_1,\gamma_2,\gamma_3,\gamma_4\}$ with $\gamma_j\equiv j\mod 5$. Then the matrix of the Fourier transform is 

$$\F=\frac{1}{\sqrt 5}\left(e^{2\pi \frac {jk}5}\right)_{j,k=0}^4.$$
The local translation matrix is
$$B=\begin{pmatrix}
0&1&0&0&0\\
0&0&1&0&0\\
0&0&0&1&0\\
0&0&0&0&1\\
1&0&0&0&0\end{pmatrix}.$$
In this case $\Theta_B=\bz$. 
\begin{acknowledgements}
This work was partially supported by a grant from the Simons Foundation (\#228539 to Dorin Dutkay).
\end{acknowledgements}

\bibliographystyle{alpha}	
\bibliography{eframes}

\end{document}